\documentclass[11pt, a4paper]{amsart}
\usepackage{amssymb, array, amsmath, amscd, pdfpages, 
enumerate, amsthm, mathtools, overpic, xcolor, dsfont, amsrefs}
\usepackage[T1]{fontenc}
\usepackage[colorlinks=true,allcolors=magenta]{hyperref}
\mathtoolsset{showonlyrefs=true}

\newtheorem{theorem}{Theorem}[section]
\newtheorem{proposition}[theorem]{Proposition}
\newtheorem{lemma}[theorem]{Lemma}

\theoremstyle{definition}

\newtheorem{remark}[theorem]{Remark}
\numberwithin{equation}{section}

\newcommand{\R}{\mathbb{R}}
\newcommand{\C}{\mathbb{C}}
\newcommand{\B}{\mathcal{B}}

\DeclareMathOperator{\re}{Re}

\DeclareMathOperator{\Ran}{Ran}

\title[Klein--Gordon equations with Kelvin--Voigt damping]
{Optimal energy decay rates for Klein--Gordon equations with Kelvin--Voigt damping}

\author[F. Dell'Oro, L. Paunonen, D. Seifert]
{Filippo Dell'Oro, Lassi Paunonen and David Seifert}

\address[F. Dell'Oro]{Dipartimento di Matematica, Politecnico di Milano
\newline\indent
Via Bonardi 9, 20133 Milano, Italy}
\email{filippo.delloro@polimi.it}

\address[L. Paunonen]{Mathematics Research Centre, Tampere University
\newline\indent P.O.\ Box 692, 33101 Tampere, Finland}
\email{lassi.paunonen@tuni.fi}

\address[D. Seifert]{School of Mathematics, Statistics and Physics,
\newline\indent Newcastle University
\newline\indent
Herschel Building, Newcastle upon Tyne, NE1 7RU, United Kingdom}
\email{david.seifert@ncl.ac.uk}

\begin{document}

\begin{abstract}
We study the long-time behaviour of solutions to 
a one-dimensional linear  Klein--Gordon equation with Kelvin--Voigt damping.
One of the interesting features of the equation is that the generator
of the associated $C_0$-semigroup has
multiple spectral points on the imaginary axis. 
As our main result, we show that the energy of every possible solution converges to zero as time goes to infinity and, moreover, we provide an optimal polynomial energy decay rate for a certain class of solutions.
\end{abstract}


\maketitle

\section{Introduction}

\noindent
In this paper we study the long-time behaviour of solutions to the following one-dimensional Klein--Gordon equation with Kelvin--Voigt damping:
\begin{equation}
\label{eq1}
\begin{cases}\begin{aligned}  
 u_{tt}(x,t) &- u_{xx}(x,t) +m u(x,t)- u_{t xx}(x,t)=0, &&\,\,\, x \in \R,\ t>0,\\
 u(x,0)&=u_0(x), &&\,\,\, x \in \R, \\
 u_t(x,0) &= v_0(x), &&\,\,\, x \in \R,
\end{aligned}\end{cases}
\end{equation}
where $m>0$ is a fixed constant.
The initial data satisfy $u_0 \in H^1(\R)$ and $v_0 \in L^2(\R)$. Setting $z_0=(u_0,v_0)$,
we define the energy of the corresponding solution as
\begin{equation}
\label{en-def}
E_{z_0}(t) = \frac12 \int_\R m|u(x,t)|^2 + |u_x(x,t)|^2 + |u_t(x,t)|^2 dx, \quad\, t\geq0.
\end{equation}
For sufficiently regular solutions, a straightforward calculation yields the energy identity
$$
\dot{E}_{z_0}(t)
= - \int_\R |u_{tx}(x,t)|^2 dx,\quad\, t\geq0,
$$
so that $E_{z_0}$ is non-increasing. The main aim of this article is to show that $E_{z_0}(t)$ actually
decays to zero as $t\to\infty$ for all possible solutions and, in addition, to find optimal
estimates for the rate of decay for certain classes of solutions.  

When considering Klein--Gordon equations with Kelvin--Voigt damping posed on a finite interval (or more generally
on a bounded multi-dimensional domain) with Dirichlet boundary conditions, it is well known that 
the energy decays exponentially to zero for all possible solutions, that is,
the associated contraction $C_0$-semigroup is exponentially stable. Indeed, in this situation,
one can take advantage of the Poincar\'e inequality and achieve exponential stability 
by following the same strategy employed for the viscous global damping, namely,
by constructing a perturbed energy functional which satisfies a convenient differential inequality 
and then applying the Gr\"onwall lemma.
Several results on the stability (and regularity) properties of linear wave equations with 
Kelvin--Voigt damping on bounded domains may be found for instance in \cite{ROBBIANO,BURQ1,BURQ2,CHENTRI,CHEN,LIU,SUN,WEB}. 

When studying problem \eqref{eq1}, or more generally
Klein--Gordon equations with Kelvin--Voigt damping
on unbounded multi-dimensional domains whose geometries are incompatible with the validity of the Poincar\'e inequality, the picture changes drastically. 
Indeed, as we will show in the present article, 
problem \eqref{eq1} is associated with a contraction $C_0$-semigroup $(T(t))_{t\geq0}$ on the Hilbert space
$X=H^1(\R)\times L^2(\R)$ whose infinitesimal generator $A$ possesses multiple (but a finite number of) spectral points on the imaginary axis.
More precisely, the intersection of the spectrum of $A$ with the imaginary axis consists exactly of the two points
$\pm i\sqrt{m}$.
This feature prevents $(T(t))_{t\geq0}$ from being exponentially stable or even semi-uniformly
stable in the sense of \cite{BatDuy08,BorTom10}. However,
since these spectral points are not eigenvalues of $A$, we may exploit 
the classical Arendt--Batty--Lyubich--V\~u theorem~\cite{AreBat88} and deduce that the energy
$E_{z_0}(t)$ converges to zero as $t\to\infty$ for all initial data $z_0\in X$. On the other hand, as already noted, with a view to finding estimates for the rate of decay of certain classes of solutions,
the presence of two imaginary spectral points means that we cannot directly exploit the classical methods of semi-uniform stability \cite{BatDuy08,BorTom10} or their generalisation \cite{BatChi16,RSS} which consider the case where the spectrum of $A$ intersects the imaginary axis at most at zero. The abstract criteria from \cite{MAR}
accommodate finite spectral points on the imaginary axis but, although they apply in the more general setting of Banach spaces, they provide only sub-optimal decay rates in Hilbert spaces.
Instead, our strategy consists in combining the
results of \cite{BatChi16} with those of \cite{Pau14c}, finding 
in our main Theorem~\ref{thm:EnergyDecay} the optimal decay estimate 
$$E_{z_0}(t)= O(t^{-2}),\quad\,\,\text{as }t\to\infty,$$ 
for a certain class of solutions to problem \eqref{eq1}. These solutions are precisely those that originate 
from initial data $z_0$ lying in the intersection of the ranges of the operators $i\sqrt{m}-A$ and $-i\sqrt{m}-A$. 

\smallskip
\begin{table}[htpb]
    \centering
    \renewcommand{\arraystretch}{1.5}
    \begin{tabular}{|p{0.23\textwidth}||p{0.29\textwidth}|p{0.30\textwidth}|}
        \hline
        {Spatial domain} & {{Viscous damping $u_t$}} & {Kelvin--Voigt damping}\\
        \hline\hline
        {Bounded interval} & 
        {Exponential decay for all initial data}  & 
        {Exponential decay for all initial data}\\
        \hline
        {Whole line $\R$} & 
        {Exponential decay for all initial data}  & 
        {Optimal polynomial decay of order $t^{-2}$ for selected initial data}\\
        \hline
    \end{tabular}
    \vspace{0.4cm}
    \caption{Energy decay for the linear Klein--Gordon equation for different damping mechanisms and spatial domains.}
\end{table}

We conclude by mentioning that various results dealing with energy decay of Klein--Gordon equations on unbounded domains with viscous damping have appeared in the 
literature:\ focusing exclusively on papers treating linear models, and without any claim to completeness,
we refer the reader to \cite{BJ,GREEN,KOTA,STANI,ROY,WANG,WANG2,JARED}. Moreover, asymptotic profiles for solutions to wave and Klein--Gordon equations on unbounded domains with different damping mechanisms, including the Kelvin--Voigt damping, have been extensively studied; see e.g.\ \cite{BA,IKECHEN,IKEDA,IKE,IKE2,IKETA,IKETODO,LIZ,PISKE}, to give just a small selection. We emphasise in particular that the exact model~\eqref{eq1} in $\R^n$ with $n\geq1$ has previously been investigated in \cite{IKEDA}, where the authors employed Fourier analysis techniques to derive asymptotic profiles for the solutions.
However, our decay results differ substantially from those obtained in \cite{IKEDA}, as the latter require the initial data to belong to an appropriate weighted $L^1$-space on $\mathbb{R}^n$ in addition to the energy space $X$.

\smallskip
The article is organised as follows. In the next Section \ref{S2}
we recast \eqref{eq1} as an abstract 
Cauchy problem on $X$
and we establish existence of the associated contraction $C_0$-semigroup $(T(t))_{t\geq0}$. We also 
prove the aforementioned description of the boundary spectrum of
the semigroup generator $A$. In the subsequent Section \ref{S3} we establish upper bounds for
the norm of the resolvent operator along the imaginary axis, both near the two spectral points $\pm i\sqrt{m}$
and at infinity. Section \ref{S4} is devoted to the proof of our main result.  
Finally, in Section~\ref{S5}, we briefly discuss the case when the equation is posed on the
half-line with a Dirichlet boundary condition, explaining how our techniques can be adapted to treat that situation.
We emphasise that our abstract operator-theoretic approach is versatile and can in principle be applied to a variety of problems,
once the appropriate resolvent estimates on the imaginary axis are established as in Theorem \ref{res-est}. For instance, in addition
to the half-line case, we expect our method to extend readily to the setting of Klein--Gordon equations with localised damping, thus differentiating it from the Fourier-analytic approach.

\medskip
\noindent
{\bf Notation.} 
The notation used is standard throughout.
If $A$ is a closed linear operator on a (complex) Hilbert space, we denote its 
domain by $D(A)$, its range by $\Ran(A)$, its spectrum by $\sigma(A)$, its point spectrum by $\sigma_p(A)$,
its continuous spectrum by $\sigma_c(A)$ and its resolvent set by $\rho(A)$. 
For $\lambda\in\rho(A)$ we write $R(\lambda,A)$ for the resolvent operator $(\lambda-A)^{-1}$. 
The symbol $\B(X)$ stands for the space of bounded linear operators on a Banach space $X$. 
We denote by $\C_{-}$ and $\C_{+}$ the open left and right complex half-plane, respectively, 
and we set $\R_+=[0,\infty)$.
Given $\lambda\in\C$, we define the square root $\sqrt{\lambda}$ by taking the branch cut along the
negative real axis. In particular,
$\re\sqrt{\lambda}\geq0$ for all $\lambda\in\C$, with strict inequality for $\lambda\not\in(-\infty,0]$.
For $p,q\in\R$
we write $p\lesssim q$ to indicate that $p\leq Cq$ 
for some constant $C>0$.
Finally, we use conventional asymptotic notation, including `big O' and `little o'.

\medskip
\section{The Semigroup Generator and Its Boundary Spectrum}
\label{S2}

\noindent
We consider the complex Hilbert space 
$X=H^1(\R)\times L^2(\R),$
endowed with the norm 
$$\|z\|_X^2 = m\|u\|_{L^2(\R)}^2 + \|u'\|_{L^2(\R)}^2 + \|v\|_{L^2(\R)}^2$$
for  $z=(u,v)\in X$. 
Then, we rewrite \eqref{eq1} as an abstract Cauchy problem on $X$, that is
\begin{equation}
\label{eq:ACP}
\left\{
\begin{aligned}
\dot{z}(t) &= A z(t),\quad t\ge0,\\
z(0)&=(u_0,v_0)\in X,
\end{aligned}\right.
\end{equation}
where $A: D(A) \subseteq X \to X$ is the linear operator $Az = (v, (u+v)''-m u)$
with domain
$$
D(A) = \big\{ (u,v)\in H^1(\R)\times H^1(\R) : u+v \in H^2(\R)\big\}.
$$
The following result summarises the main properties of $A$, including a description of
its spectrum on the imaginary axis $i\R$.

\begin{theorem}
\label{theo-A}
The following hold:
\begin{itemize}
\vspace{0.1mm}
\item[(a)] $A$ is closed;

\vspace{0.4mm}
\item[(b)] $A$ is dissipative;

\vspace{0.4mm}
\item[(c)] $\sigma(A) \cap i\R =\{i\sqrt{m},-i\sqrt{m}\}$;

\vspace{0.4mm}
\item[(d)] $\sigma_p(A) \cap i\R =\emptyset$.

\end{itemize}
\end{theorem}

\begin{proof}
(a) Closedness of $A$ may be verified by means of standard methods, for instance by adapting  the argument given in the proof of \cite[Lemma 2.1]{NgSei20}.

\smallskip
\noindent
(b) For every $z = (u,v)\in D(A)$, a straightforward calculation yields
$$
\re \langle A z,z\rangle_X = -\|v'\|_{L^2(\R)}^2\leq0,
$$
so $A$ is dissipative. 

\smallskip
\noindent
(c) We begin by proving that $i\R \setminus \{i\sqrt{m},-i\sqrt{m}\}\subseteq\rho(A)$.
To this end, let us fix $s\in\R$ with $|s| \neq \sqrt{m}$. Since $A$ is closed,
it is enough to show that for every $\hat z=(\hat u,\hat v)\in X$ the resolvent equation
\begin{equation}
\label{res-im}
(i s - A) z = \hat z
\end{equation}
has a unique solution $z = (u,v)\in D(A)$. 
The equation is equivalent to
\begin{align*}
i s u(x) - v(x)= \hat u(x),\\
i s v(x) - (u+v)''(x) +m u(x)= \hat v(x). 
\end{align*}
Setting $w=u+v$ 
we may rewrite the system above in the form
\begin{subequations}
\begin{align}
\label{sist1}
(1+i s) u(x) - w(x)= \hat u(x),\\
\label{sist2}
(m-s^2) u(x) - w''(x) = i s \hat u(x) + \hat v(x).
\end{align}
\end{subequations}
Substituting \eqref{sist1} into \eqref{sist2}, we obtain
\begin{equation}
\label{res}
- w''(x) +\lambda_s\hspace{0.3mm} w(x)= \hat v(x)  +(is- \lambda_s)\hat u(x), 
\end{equation}
where
\begin{equation}
\label{deff}
\lambda_s = \frac{m-s^2}{1+is}\not\in(-\infty,0].
\end{equation}
To solve \eqref{res}, we introduce the 
Green's function 
\begin{equation}
\label{defGreen}
G_s(x) = \frac{1}{2 \sqrt{\lambda_s}} \, {\rm e}^{-\sqrt{\lambda_s} |x|}.
\end{equation}
Since ${\rm Re} \sqrt{\lambda_s}>0$, 
it is readily seen that
$G_s \in L^1(\R)$. 
Denoting by $*$ the convolution product on $\R$, 
the solution to \eqref{res} 
may be written as
$$
w(x) = (G_s \,*\, \hat v)(x) + (is- \lambda_s)(G_s \,*\, \hat u)(x).
$$
Indeed, by means of a simple calculation, we obtain
$$
w'(x) = (Q_s \,*\, \hat v)(x) + (is- \lambda_s)(Q_s \,*\, \hat u)(x),
$$
where
\begin{equation}
\label{defQ}
Q_s(x) =\frac12\big[{\rm e}^{\sqrt{\lambda_s}x} H(-x) - {\rm e}^{-\sqrt{\lambda_s}x} H(x)\big]
\end{equation}
with $H$ being the Heaviside step function.
Note that $w,w'\in L^2(\R)$ by
Young’s inequality for convolutions and the fact that $G_s,Q_s\in L^1(\R)$. A further simple calculation
now shows that $w$ solves \eqref{res} and so in particular $w\in H^2(\R)$.
Once $w=u+v$ has been found, it is immediate to check that 
\begin{equation}
\begin{cases}
\label{u-form}
 u(x)= \displaystyle\frac{1}{1+is}\big[(G_s \,*\, \hat v)(x) + (is- \lambda_s)(G_s \,*\, \hat u)(x)
+ \hat u(x)\big]\\
\noalign{\vskip2.3mm}
v(x)= \displaystyle \frac{1}{1+is}\big[is(G_s \,*\, \hat v)(x) -(s^2+is\lambda_s)(G_s \,*\, \hat u)(x)-\hat u(x)\big]
\end{cases}
\end{equation}
is the desired unique solution to \eqref{res-im}.

It remains to show that $\pm i \sqrt{m}\in\sigma(A)$.
We will work with $i \sqrt{m}$ only; an analogous argument applies to 
$-i \sqrt{m}$. For $k\geq1$, let
$$
u_k(x) = \frac{1}{\sqrt{k}} \Phi\Big(\frac{x}{k}\Big), 
$$
where $\Phi:\R \to [0,1]$ is a smooth bump function 
with $\Phi(x)=1$ for $|x|\leq 1/2$ and $\Phi(x)=0$ for $|x|\geq1$. Since $u_k \in H^2(\R)$
it follows immediately that $z_k=(u_k,i \sqrt{m}u_k)\in D(A)$ for every $k\geq1$. Moreover,
$$
\|z_k\|_X^2 \geq 2m \|u_k\|^2_{L^2(\R)} 
= \frac{2m}{k}\int_\R \Big|\Phi\Big(\frac{x}{k}\Big) \Big|^2  dx \geq  2m.
$$ 
On the other hand, it is straightforward to check that
$$
\|(i\sqrt{m}-A)z_k\|_X^2 = (1+m)\|u_k''\|^2_{L^2(\R)} 
\leq \frac{2(1+m)}{k^4}\|\Phi''\|_{L^\infty(\R)}^2
\to 0
$$
as $k\to\infty$.
Hence $i\sqrt{m}$ is an approximate eigenvalue of $A$.

\smallskip
\noindent
(d) If $(i\sqrt{m}-A)z=0$ for some $z=(u,v)\in D(A)$, reasoning 
as above in the passage from \eqref{res-im}
to \eqref{res} 
we see that $w=u+v$ satisfies the equation $w''(x)=0$. 
Since $w\in H^2(\R)$, this forces $w=0$ and therefore $z=0$ as well.
Hence $i\sqrt{m}\notin\sigma_p(A)$.
An analogous argument applies to 
$-i \sqrt{m}$.
\end{proof}

Note that Theorem \ref{theo-A} tells us in particular that $0\in\rho(A)$.
Since the resolvent set is open, this implies that $\lambda-A$ is surjective for some
$\lambda>0$.
Since $A$ is also dissipative, exploiting the Lumer-Phillips theorem we obtain the following result;
see e.g. \cite[Corollary II.3.20]{ENGNAG}.

\begin{theorem}
\label{theo-gen}
The operator $A$ is densely defined and generates a contraction $C_0$-semigroup $(T(t))_{t\geq0}$ on the Hilbert space $X$. 
\end{theorem}

\section{Resolvent Estimates}
\label{S3}

\noindent
In this section, we obtain upper bounds on the growth rate of $\|R(is,A)\|$ as $|s|$
tends to $\sqrt{m}$ and to infinity. Our main result here is the following theorem.

\begin{theorem}
\label{res-est}
The following hold:
\begin{itemize}

\item[(a)] $\|R(is,A)\| = O(1)$ as $|s|\to\infty$;

\vspace{0.7mm}
\item[(b)] $\|R(is,A)\| = O\big(|s-\sqrt{m}|^{-1}\big)$ as $s\to\sqrt{m}$;

\vspace{0.7mm}
\item[(c)] $\|R(is,A)\| = O\big(|s+\sqrt{m}|^{-1}\big)$ as $s\to-\sqrt{m}$.
\end{itemize}
\end{theorem}

In what follows, for $s\in\R$ with $|s|\neq\sqrt{m}$, the
number $\lambda_s\in\C\setminus (-\infty,0]$ is defined in \eqref{deff}.
We begin with an elementary technical lemma.

\begin{lemma}
\label{tech}
There exists a constant $m_0>0$ depending only on 
the (fixed) parameter $m>0$ such that 
$$
\cos\Big(\frac{\arg \lambda_s}{2}\Big)\geq m_0
$$
for all $s\in\R$ with $|s|\neq\sqrt{m}$.
\end{lemma}

\begin{proof}
For $|s| < \sqrt{m}$ we have $\arg \lambda_s = - \arctan s$,
yielding $|\arg \lambda_s| < \frac{\pi}{2}$. Therefore
$$\cos\Big(\frac{\arg \lambda_s}{2}\Big)\geq\frac{1}{{\sqrt{2}}}.$$
For $ s > \sqrt{m}$ we have $\arg \lambda_s = - \arctan s + \pi$, and thus
$\frac{\pi}{2}< \arg \lambda_s < - \arctan \sqrt{m} + \pi$. As a consequence
$$
\cos\Big(\frac{\arg \lambda_s}{2}\Big) \geq \cos\Big(- \frac{\arctan \sqrt{m}}{2} + \frac{\pi}{2}\Big)
= \sin \Big(\frac{\arctan \sqrt{m}}{2}\Big)>0.
$$
Finally, for $ s <-\sqrt{m}$ we have $\arg \lambda_s = - \arctan s - \pi$, which implies that
$\arctan \sqrt{m} - \pi< \arg \lambda_s < -\frac{\pi}{2}$
and the conclusion follows as before.
\end{proof}

\begin{proof}[Proof of Theorem \ref{res-est}]
Let $s\in\R$ with $|s|\neq\sqrt{m}$ and $\hat z=(\hat u,\hat v)\in X$ be fixed.
From Theorem \ref{theo-A} and its proof, we know that the resolvent
equation $(i s - A) z = \hat z$ has a unique solution $z = (u,v)=R(is,A)\hat z\in D(A)$
given by~\eqref{u-form}. Recall also that the Green's function
$G_s\in L^1(\R)$ is defined in~\eqref{defGreen}.

We first estimate $\|u\|_{L^2(\R)}$. To this end,
using Young’s inequality for convolutions, we compute
\begin{align*}
\|u\|_{L^2(\R)} \lesssim \frac{1}{{1+|s|}}\|G_s\|_{L^1(\R)}\big[\|\hat v\|_{L^2(\R)}
+|is-\lambda_s|\|\hat u\|_{L^2(\R)}\big]
 + \frac{1}{{1+|s|}}\|\hat u\|_{L^2(\R)},
\end{align*}
where, as in the remainder of the proof, the implicit constant depends on 
the (fixed) parameter $m>0$ but is independent
of both $s$ and $\hat z$. Noting that $|is-\lambda_s|\lesssim 1+|s|$, we find
\begin{align*}
\|u\|_{L^2(\R)} &\lesssim \|G_s\|_{L^1(\R)}\big[\|\hat v\|_{L^2(\R)}+ \|\hat u\|_{L^2(\R)}\big]
 + \frac{1}{{1+|s|}}\|\hat u\|_{L^2(\R)}\\
&\lesssim \|G_s\|_{L^1(\R)}\|\hat z\|_X +\frac{1}{{1+|s|}} \|\hat z\|_X.
\end{align*}
At this point, a straightforward calculation combined with Lemma \ref{tech} shows that
$$
\|G_s\|_{L^1(\R)} = \frac{1}{|\lambda_s|}\Big[\cos\Big(\frac{\arg \lambda_s}{2}\Big)\Big]^{-1}
\lesssim \frac{1}{|\lambda_s|} \lesssim \frac{1+ |s|}{|s-\sqrt{m}||s+\sqrt{m}|},
$$
yielding
$$
\|u\|_{L^2(\R)} \lesssim \frac{1+ |s|}{|s-\sqrt{m}||s+\sqrt{m}|} \|\hat z\|_X.
$$
Since $v = i s u - \hat u$, from the bound above we easily get
$$
\|v\|_{L^2(\R)} \lesssim \frac{1+ s^2}{|s-\sqrt{m}||s+\sqrt{m}|} \|\hat z\|_X.
$$
Next, by means of an elementary calculation, we see that
$$
u'(x)= \frac{1}{1+is}\big[(Q_s \,*\, \hat v)(x) + (is- \lambda_s)(Q_s \,*\, \hat u)(x)
+ \hat u'(x)\big],
$$ 
where we recall that the function
$Q_s\in L^1(\R)$ is defined in \eqref{defQ}.
As before, using Young’s inequality for convolutions and recalling that
$|is-\lambda_s|\lesssim 1+|s|$, we obtain
$$
\|u'\|_{L^2(\R)} \lesssim\|Q_s\|_{L^1(\R)}\|\hat z\|_X+\frac{1}{{1+|s|}}\|\hat z\|_X.
$$
An application of Lemma \ref{tech} combined with a straightforward calculation gives
$$
\|Q_s\|_{L^1(\R)}= \frac{1}{|\lambda_s|^{1/2}}\Big[\cos\Big(\frac{\arg \lambda_s}{2}\Big)\Big]^{-1}
\lesssim \frac{1}{|\lambda_s|^{1/2}} 
\lesssim \frac{1+ |s|^{1/2}}{|s-\sqrt{m}|^{1/2}|s+\sqrt{m}|^{1/2}},
$$
and thus
$$
\|u'\|_{L^2(\R)} 
\lesssim \frac{1+ |s|^{1/2}}{|s-\sqrt{m}|^{1/2}|s+\sqrt{m}|^{1/2}}\|\hat z\|_X.
$$
Collecting the bounds obtained so far, we see that 
$$
\|z\|_{X} \lesssim \frac{1+ s^2}{|s-\sqrt{m}||s+\sqrt{m}|} \|\hat z\|_X,
$$
which implies that
$$
\|R(is,A)\| \lesssim \frac{1+ s^2}{|s-\sqrt{m}||s+\sqrt{m}|}.
$$
The result follows.
\end{proof}

\section{Asymptotic Stability and Quantified Decay Rates}
\label{S4}

\noindent
The following is the main result of the paper.

\begin{theorem}
\label{thm:EnergyDecay}
The energy \eqref{en-def} satisfies $E_{z_0}(t)\to 0$ as $t\to\infty$ for every $z_0\in X$.
Moreover, 
$$E_{z_0}(t)= O(t^{-2}), \quad\,\, \text{as }t\to\infty,$$ 
for every $z_0\in \Ran(i\sqrt{m}-A)\cap \Ran(-i\sqrt{m}-A)$. This decay rate is
optimal in the sense that given any function $r:\R_+\to(0,\infty)$ satisfying $r(t)=o(t^{-2})$ as $t\to\infty$, there exists $z_0\in \Ran(i\sqrt{m}-A)\cap \Ran(-i\sqrt{m}-A)$ such that $E_{z_0}(t)\ne O(r(t))$ as $t\to\infty$.
\end{theorem}

\begin{proof}
Since $\sigma_p(A)\cap i\mathbb{R}=\emptyset$ and $\sigma(A)\cap i\R
=\{i\sqrt{m},-i\sqrt{m}\}$ is finite by Theorem~\ref{theo-A}, it follows from the 
classical Arendt--Batty--Lyubich--V\~u stability theorem~\cite{AreBat88} that
the contraction $C_0$-semigroup $(T(t))_{t\geq0}$ generated by $A$ is strongly asymptotically stable,
i.e.\ 
$$E_{z_0}(t)=\frac12 \|T(t)z_0\|_X^2\to 0,\quad\,\, \text{as }t\to\infty,$$ for 
all $z_0\in X$.
Next,  let $B_1 = (i \sqrt{m}-A)R(1,A)\in \B(X)$ and $B_2 = (-i \sqrt{m}-A)R(1,A)\in \B(X)$.
Note that $B_1$ and $B_2$ commute, and that $B_1B_2$ commutes with $A$.
By Theorems~\ref{theo-A} and~\ref{res-est} we have
\begin{align*}
|s-\sqrt{m}|\|R(is,A)\|\lesssim 1, \qquad&
s\in [0,\sqrt{m})\cup(\sqrt{m},2 \sqrt{m}],\\\noalign{\vskip0.3mm}
|s+ \sqrt{m}|\|R(is,A)\|\lesssim 1, 
\qquad& s\in [-2 \sqrt{m},-\sqrt{m})\cup (-\sqrt{m},0],\\\noalign{\vskip0.3mm}
\|R(is,A)\|\lesssim 1, \qquad& s\in \R\setminus [-2 \sqrt{m},2 \sqrt{m}],
\end{align*}
\noindent
and hence Assumption~3 in~\cite{Pau14c} is satisfied.
As a consequence, if we define 
$\Omega_\pm = \{\lambda\in\overline{\C_+}\, : 0<|\lambda\mp i\sqrt{m}|\le \min\{1,\sqrt{m}/2\}\}$
then~\cite[Lemma~13]{Pau14c} implies that 
\begin{align*}
\sup_{\lambda \in \Omega_\pm}  \| R(\lambda,A)(\pm i\sqrt{m}-A)R(1\pm i\sqrt{m},A)\|<\infty.
\end{align*}
Boundedness of $B_1$ and $B_2$ combined with the resolvent identity then yield
$\sup_{\lambda \in \Omega_+\cup \Omega_-}  \| R(\lambda,A)B_1B_2\|<\infty$.
Since~\cite[Lemma~15]{Pau14c} guarantees that 
$\sup_{\lambda\in\C_+\setminus (\Omega_+\cup \Omega_-)} \|R(\lambda,A)\|<\infty,$
 we obtain
$\sup_{\lambda\in \mathbb{C}_+}\| R(\lambda , A)B_1B_2\| <\infty.$
The latter bound and~\cite[Theorem 4.7]{BatChi16} 
lead to the estimate
\begin{equation}
\label{bou}\|T(t)B_1B_2\|^2=O(t^{-2}),\quad\,\, \text{as }t\to\infty.
\end{equation}
We now show that $\Ran(B_1B_2)=\Ran (i\sqrt{m}-A)\cap \Ran (-i\sqrt{m}-A)$.
Clearly $\Ran(B_1)=\Ran (i\sqrt{m}-A)$ and $\Ran(B_2)=\Ran (-i\sqrt{m}-A)$. It is also clear that
\begin{align*}
\Ran(B_1B_2)\subseteq 
\Ran(B_1)\cap \Ran(B_2) =
\Ran (i\sqrt{m}-A)\cap \Ran (-i\sqrt{m}-A).
\end{align*}
On the other hand, if 
$w\in  \Ran(B_1)\cap \Ran(B_2) $, 
then $w=B_1z$ for some $z\in X$, and so
$
w=B_1z = 2i \sqrt{m}R(1,A) z + B_2 z.
$
Recalling that $w\in \Ran(B_2)$, the identity above yields $R(1,A) z\in \Ran(B_2)$,
which is to say that  $R(1,A)z = B_2 y$ for some $y\in X$. Since $B_2 = (-i \sqrt{m}-1) R(1,A) + I$, we have $y\in D(A)$ and $z = B_2(I-A)y$, which finally implies that 
$$w=B_1B_2(I-A)y\in \Ran(B_1B_2).$$
This proves that $\Ran(B_1B_2)=\Ran (i\sqrt{m}-A)\cap \Ran (-i\sqrt{m}-A)$.
In light of \eqref{bou} and the equality $E_{z_0}(t)=\frac12 \|T(t)z_0\|_X^2$ for $z_0\in X$ and $t\ge0$, the proof
of the first part of Theorem \ref{thm:EnergyDecay} is complete.

For the optimality part, let $r:\R_+\to(0,\infty)$ be such that $r(t)=o(t^{-2})$ as $t\to\infty$ and suppose, for the sake of obtaining a contradiction, that
$E_{z_0}(t)= O(r(t))$ as $t\to\infty$ for all $z_0\in \Ran(i\sqrt{m}-A)\cap \Ran(-i\sqrt{m}-A)$. Since $\Ran (i\sqrt{m}-A)\cap \Ran (-i\sqrt{m}-A)=\Ran(B_1B_2)$, an application of the uniform boundedness principle yields $\|T(t)B_1B_2\|^2=O(r(t))$ as $t\to\infty$, and hence 
$t\|T(t)B_1B_2\|\to0$ as $t\to\infty$.
We now adapt the argument given in the proof of~\cite[Theorem~6.9]{BatChi16}. Suppose first that $ i\sqrt{m}$ is a  limit point of $\sigma(A)$. Then there exists a sequence $(\lambda_k)_{k\ge1}$ in $\sigma(A)\cap\C_-$ such that $\lambda_k\to i \sqrt{m}$ as $k\to\infty$. Let $t_k=-1/\re\lambda_k$ for $k\ge1$, noting that $t_k\to\infty$ as $k\to\infty$. 
Recalling that the spectral radius of a bounded linear
operator is dominated by the norm of the operator, it
follows from the spectral inclusion theorem for the Hille--Phillips functional calculus \cite[Theorem 16.3.5]{HPBOOK} that
$$\begin{aligned}
\limsup_{t\to\infty}\,t\|T(t)B_1B_2\|&\ge\limsup_{k\to\infty}\, t_k \left| e^{t_k\lambda_k}\frac{ (i \sqrt{m}-\lambda_k)(-i \sqrt{m}-\lambda_k)}{(1-\lambda_k)^2}\right| \\
&\ge\frac1e\lim_{k\to\infty}  \left| \frac{ i \sqrt{m}+\lambda_k}{(1-\lambda_k)^2}\right|=\frac{2\sqrt{m}}{e(1+m)}>0,
\end{aligned}$$
which contradicts our earlier observation that $t\|T(t)B_1B_2\|\to0$ as $t\to\infty$. 
So $i \sqrt{m}$ must be an isolated point of $\sigma(A)$. From the theory of spectral decompositions coming out of the Riesz--Dunford functional calculus, we therefore obtain a decomposition $X=X_0\oplus X_1$ of our space $X$ into two non-trivial closed subspaces $X_0, X_1$ which are invariant under $(T(t))_{t\ge0}$ and such that, furthermore, the restrictions $A_k$ of $A$ to $X_k$, $k=0,1$, satisfy $A_1\in\B(X_1)$, $\sigma(A_1)=\{i \sqrt{m}\}$ and $\sigma(A_0)=\sigma(A)\setminus\{i \sqrt{m}\}$;
see e.g.\ \cite[Chapter~IV]{ENGNAG}. 
Let $(T_1(t))_{t\ge0}$ denote the semigroup on $X_1$ generated by $A_1-i \sqrt{m}$. Then
$$\begin{aligned}
&\liminf_{t\to\infty}\, t\|(i \sqrt{m}-A_1)T_1(t)\|\\\le& \limsup_{t\to\infty}t\|T_1(t)(i \sqrt{m}-A_1)(-i \sqrt{m}-A_1)R(1,A_1)^2\|\|R(-i \sqrt{m},A_1)(I-A_1)^2\|\\
\le&\limsup_{t\to\infty}t\|T(t)B_1B_2\|\|R(-i \sqrt{m},A_1)(I-A_1)^2\| =0,
\end{aligned}$$
and it follows from~\cite[Theorem~2.1]{KalMon04} that $i \sqrt{m}-A_1=0$. This contradicts the fact that $i \sqrt{m}\not\in\sigma_p(A)$, so the argument is complete.
\end{proof}

We conclude this main section by analysing in more detail
the intersection $\Ran(i\sqrt{m}-A)\cap \Ran(-i\sqrt{m}-A)$ in order to better understand 
the class of solutions for which we have our decay estimate. First, we prove that this set is dense in the state
space.

\begin{proposition}\label{prp:dense}
The set  $\Ran(i\sqrt{m}-A)\cap \Ran(-i\sqrt{m}-A)$ is dense in $X$.
\end{proposition}

\begin{proof}
Since $A$ generates a contraction $C_0$-semigroup on a Hilbert space, and since
$\pm i\sqrt{m} \in \sigma(A) \setminus \sigma_{{p}}(A)$ by Theorem \ref{theo-A},
it follows that
$\pm i\sqrt{m} \in \sigma_{{c}}(A)$; see e.g.\ \cite[Proposition 2.2]{AreBat88}.
Setting $B_1 = (i \sqrt{m}-A)R(1,A)\in \B(X)$ and $B_2 = (-i \sqrt{m}-A)R(1,A)\in \B(X)$
as in the proof of Theorem \ref{thm:EnergyDecay}, we thus infer that
$\Ran(i\sqrt{m}-A)= \Ran(B_1)$ and $\Ran(-i\sqrt{m}-A) = \Ran(B_2)$ are dense in $X$,
which implies that $\Ran(B_1B_2)$ is dense in $X$.
Recalling that  
$\Ran (i\sqrt{m}-A)\cap \Ran (-i\sqrt{m}-A)=\Ran(B_1B_2)$
as we showed in the proof of Theorem \ref{thm:EnergyDecay}, the conclusion follows.
\end{proof}

\begin{remark}
Proposition~\ref{prp:dense} allows us to remove the appeal to the Arendt--Batty--Lyubich--V\~u stability theorem in the proof of the first part of Theorem~\ref{thm:EnergyDecay}. Indeed, once we have established that $\|T(t)z_0\|\to0$ as $t\to\infty$ for all $z_0\in \Ran(i\sqrt{m}-A)\cap \Ran(-i\sqrt{m}-A)$ then it follows by means of a standard approximation argument using Proposition~\ref{prp:dense} and uniform boundedness of the semigroup $(T(t))_{t\ge0}$ that $E_{z_0}(t)\to0$ as $t\to\infty$ for all $z_0\in X$. Note that this does not require us to know a \emph{rate} of decay of $\|T(t)z_0\|$ as $t\to\infty$ for $z_0\in \Ran(i\sqrt{m}-A)\cap \Ran(-i\sqrt{m}-A)$, and indeed it is sufficient that $\|T(t)B_1B_2\|\to0$ as $t\to\infty$. This unquantified decay result is a consequence of~Theorem~\ref{theo-A}, Theorem~\ref{res-est} and \cite[Theorem~6.14]{BatChi16} applied to the measure $\mu=\delta_0+\mu_{f}$, where $\delta_0$ is the Dirac measure concentrated at zero and $\mu_f$ denotes the absolutely continuous measure on $\R_+$ corresponding to the function $f\in L^1(0,\infty)$ defined by $f(t)=((1+m)t-2)e^{-t}$ for $t>0$. 
\end{remark}

Finally, in the same spirit as \cite{NgSei20}, we provide a more concrete characterisation of the elements of $\Ran(i\sqrt{m}-A)\cap \Ran(-i\sqrt{m}-A)$.
\begin{proposition}
Let $z_0=(u_0,v_0)\in X$. Then $z_0\in \Ran(i\sqrt{m}-A)\cap \Ran(-i\sqrt{m}-A)$ if and only
if all of the following conditions hold:
\begin{itemize}
\item[{(a)}] $x \mapsto \displaystyle \lim_{a\to -\infty}\int_a^x [i\sqrt{m}u_0(r)+v_0(r)] dr  \in L^2(\R)$;

\smallskip
\item[{(b)}]  $x \mapsto \displaystyle\lim_{a\to -\infty}\int_a^x [-i\sqrt{m}u_0(r)+v_0(r)] dr\in L^2(\R)$;

\smallskip
\item[{(c)}] $x \mapsto \displaystyle\lim_{b\to -\infty}
\int_b^x \Big(\lim_{a\to -\infty} \int_a^y [i\sqrt{m}u_0(r)+v_0(r)] dr\Big) dy \in L^2(\R)$;

\smallskip
\item[{(d)}] $x \mapsto \displaystyle\lim_{b\to -\infty}
\int_b^x \Big(\lim_{a\to -\infty} \int_a^y [-i\sqrt{m}u_0(r)+v_0(r)] dr\Big) dy \in L^2(\R)$.
\end{itemize}
\end{proposition}

We omit the proof of this last proposition, which can be carried out by
adapting the arguments contained in the proof of \cite[Proposition 4.3]{NgSei20}.

\section{The Half-Line Case} 
\label{S5}

\noindent 
As mentioned in the
introduction, our techniques can be
modified straightforwardly to deal with the Klein--Gordon equation on the half-line 
with Kelvin--Voigt damping and Dirichlet boundary condition:
\begin{equation}
\label{eq1half}
\begin{cases}\begin{aligned}  
 u_{tt}(x,t) &- u_{xx}(x,t) +m u(x,t)- u_{t xx}(x,t)=0, &&& x \in (0,\infty),\ t>0,\\
 u(0,t) &= 0, & && t>0, \\
 u(x,0)&=u_0(x), & && x \in (0,\infty), \\
 u_t(x,0) &= v_0(x), &&& x \in (0,\infty).
\end{aligned}\end{cases}
\end{equation}
Again $m>0$ is a fixed constant, 
while this time the initial data satisfy $u_0 \in H^1(0,\infty)$ 
and $v_0 \in L^2(0,\infty)$ with $u_0(0)=0$. 
In this situation,
we work with the Hilbert space 
$$
Y=\big\{(u,v) \in H^1(0,\infty)\times L^2(0,\infty): u(0) = 0\big\}$$
endowed with the same norm as in Section \ref{S2}, and we consider the operator $A$ 
with the new domain
$$
D(A) = \big\{ (u,v)\in Y : v \in H^1(0,\infty),\,\, v(0)=0,\,\, u+v \in H^2(0,\infty)\big\}.
$$
Theorems \ref{theo-A} and \ref{theo-gen} remain valid.
In particular, in order to show that 
$\pm i \sqrt{m}\in\sigma(A)$, 
we may use the same argument as in the proof of Theorem~\ref{theo-A}, 
but this time with
$$
u_k(x) = \frac{1}{\sqrt{k}} \Phi\Big(\frac{x}{k}-1\Big),\quad x \in [0,\infty),
$$
to accommodate the Dirichlet boundary condition. Likewise, 
for every fixed $s\in\R$ with $|s|\neq\sqrt{m}$ and $\hat z=(\hat u,\hat v)\in Y$, 
the resolvent equation $(i s - A) z = \hat z$ 
has a unique solution $z = (u,v)=R(is,A)\hat z\in D(A)$. Extending
$\hat u(x)$ and $\hat{v}(x)$ by zero for $x\leq0$, and considering the same
complex number $\lambda_s$ in \eqref{deff} and the same
Green's function $G_s$ in \eqref{defGreen}, the solution becomes 
\begin{align*}
u(x)&= \displaystyle\frac{1}{1+is}\big[(G_s \,*\, \hat v)(x)  +(is- \lambda_s)(G_s \,*\, \hat u)(x)
+ \hat u(x)-\hat \eta_s \hspace{0.2mm} G_s(x)\big]\\\noalign{\vskip1.7mm}
v(x)&= \displaystyle \frac{1}{1+is}\big[is(G_s \,*\, \hat v)(x) -(s^2+is\lambda_s)(G_s \,*\, \hat u)(x)-\hat u(x)-i s\hspace{0.2mm}\hat \eta_s \hspace{0.2mm} G_s(x)\big]
\end{align*}
for $x \in [0,\infty)$, where $\hat \eta_s\in\mathbb{C}$ is given by
$$\hat \eta_s=2\sqrt{\lambda_s}\big[(G_s \,*\, \hat v)(0)+(is- \lambda_s)(G_s \,*\, \hat u)(0)\big].$$ 
The additional term containing $\hat \eta_s$ is needed on account of the Dirichlet boundary condition.
Arguments completely analogous to those presented in
Section \ref{S3} lead to the same resolvent bounds of Theorem~\ref{res-est}, while
Theorem~\ref{thm:EnergyDecay} and its proof are unchanged.

\subsection*{Data availability statement}
Data sharing is not applicable to this paper because no dataset was analysed or generated
during the study.

\subsection*{Conflict of interest}
The authors declare no conflicts of interest.

\subsection*{Declaration of AI use}
The authors used AI tools solely for the purpose of language editing.

\subsection*{Acknowledgments}
F.\ Dell'Oro is member of the Gruppo Nazionale per
l'Analisi Mate\-ma\-tica, la Probabilit\`a e le loro Applicazioni (GNAMPA)
of the Istituto Nazionale di Alta Matematica (INdAM).
{The research of L. Paunonen was supported by the Research Council of Finland grant 349002.}

\bibliography{DePaSe26-reference.bib}
\bibliographystyle{plain}

\end{document}